\documentclass[a4,10pt]{amsart}
\usepackage{amsfonts}
\usepackage{amssymb}
\usepackage{amsmath}
\usepackage{amsthm}
\usepackage{fullpage}

\theoremstyle{plain}
\newtheorem{thm}{Theorem}[section]

\theoremstyle{definition}

\newtheorem{ex}[thm]{Example}

\numberwithin{equation}{section}

\begin{document}
\title[Elliptic Systems]{Semilinear elliptic systems with dependence on the gradient}

\author[F. Cianciaruso]{Filomena Cianciaruso}
\address{Filomena Cianciaruso, Dipartimento di Matematica e Informatica, Universit\`{a}
della Calabria, 87036 Arcavacata di Rende, Cosenza, Italy}
\email{filomena.cianciaruso@unical.it}
\author{Paolamaria Pietramala}%
\address{Paolamaria Pietramala, Dipartimento di Matematica e Informatica, Universit\`{a}
della Calabria, 87036 Arcavacata di Rende, Cosenza, Italy}
\email{pietramala@unical.it}
\subjclass[2010]{Primary 35B07, secondary 35J57, 47H10, 34B18}
\keywords{Elliptic system, annular domain, radial solution, non-existence, cone, fixed point index.}
\begin{abstract}
We provide results on the existence, non-existence, multiplicity and localization
of positive radial solutions for semilinear
elliptic systems with Dirichlet or Robin boundary conditions on an annulus.
Our approach is topological and relies on the classical fixed
point index. We present an example to illustrate our theory.
\end{abstract}
\maketitle

\section{Introduction}
The object of this paper is to study  existence or non-existence of positive radial solutions for the system of BVPs
\begin{equation}\label{ellbvp-secapp}
\begin{cases}\
&-\Delta u =f_1(|x|,u,v,|\nabla u|,|\nabla v|)\text{ in } \Omega, \\
&-\Delta v=f_2(|x|,u,v,|\nabla u|,|\nabla v|)\text{ in } \Omega,\\
&u=0 \text{ on }\partial \Omega,\\
&v=0 \text{ on }|x|=R_0 \text{ and }\displaystyle\frac{\partial
v}{\partial r}=0 \text{ on }|x|=R_1,
\end{cases}
\end{equation}
where $\Omega=\{ x\in\mathbb{R}^n : R_0<|x|<R_1\}$ is an annulus,
$0<R_0<R_1<+\infty$, the nonlinearities $f_i$ are positive
continuous functions and $\dfrac{\partial}{\partial r}$ denotes (as
in ~\cite{nirenberg}) differentiation in the radial direction
$r=|x|$.\\
The problem of the existence of positive \emph{radial}
solutions of elliptic equations with nonlinearities with dependence on gradient subject to Dirichlet or Robin
 boundary conditions has been investigated,
via different methods, by a number of authors, we mention the
papers  \cite{ave-mo-tor, bue-er-zu-fe, defig-sa-ubi, defig-ubi, fa-mi-pe, singh}.\\

Since we are looking for radial solutions, we obtain results to a
wide class of nonlinearities by means of an approach based on
fixed point index theory and invariance properties of a
suitable cone, in which Harnack-type inequalities are used. 
Radial solutions of elliptic systems have been studied with topological  methods in \cite{bue-er-zu-fe, chen-lu, defig-sa-ubi, defig-ubi, lee, lee2, singh}, in particular the cones have been used in \cite{chen-lu, lee, lee2}.\\
The use of cones and Harnack-type inequalities in the context of ordinary differential equations or systems depending on the first derivative has been studied in \cite{aga-ore-yan, ave-graef-liu, defig-sa-ubi, guo-ge, inf, janko, yang-kong,zima}.\\
The paper is organized as follows: in Section $2$ we introduce and prove our auxiliary results, in particular the properties of a suitable  compact operator, Section $3$ contains the existence and non-existence results for the elliptic sistem and an example.   
\section{Setting and Preliminaries}

Consider  in $\mathbb{R}^n$, $n\ge 2$, the equation
\begin{equation}\label{eqell}
-\triangle w=f(|x|,w,|\nabla w|)\text{ in } \Omega,\,
\end{equation}
subject to Dirichlet boundary conditions
\begin{equation*}
w=0 \text{    on }\partial \Omega,
\end{equation*}
or Robin boundary conditions
\begin{equation*}
w=0 \text{ on }|x|=R_0 \text{ and }\displaystyle\frac{\partial
w}{\partial r}=0 \text{ on }|x|=R_1.
\end{equation*}
Since we are looking for the existence of radial solutions $w=w(r)$  of the system \eqref{ellbvp-secapp},  we rewrite \eqref{eqell} in the form
\begin{equation}\label{eqinterm}
-w''(r)-\dfrac{n-1}{r}w'(r)= f(r,w(r),|w'(r)|) \quad\text{ in }
[R_{0}, R_{1}].
\end{equation}
Set $w(t)=w(r(t))$, where, for $t\in[0,1]$, (see \cite{defig-ubi, dolo3})
\begin{equation*}
r(t):=\begin{cases}
R_1^{1-t}R_0^{t},\ &n=2,\\
\left(\frac{A}{B-t}\right)^{\frac{1}{n-2}},\ &n\geq 3,
\end{cases}
\end{equation*}
$$
A=\frac{(R_0R_1)^{n-2}}{R_1^{n-2}-R_0^{n-2}}\,\,\,\,\text{
 and  }\,\,\,\,B=\frac{R_1^{n-2}}{R_1^{n-2}-R_0^{n-2}}\,.
$$
Take, for $t\in[0,1]$,
\begin{equation*}
p(t):=\begin{cases}
r^2(t) \log^2(R_1/R_0), \ & n=2,\\
\left(\frac{R_0R_1\left(R_1^{n-2}-R_0^{n-2}\right)}{n-2}\right)^2\,\frac{1}{\left(R_1^{n-2}-(R_1^{n-2}-R_0^{n-2})t\right)^{\frac{2(n-1)}{n-2}}}\,\,,\
&n\geq 3,
\end{cases}
\end{equation*}
then the equation ~\eqref{eqinterm} becomes
\begin{equation*}
-w''(t)= p(t)
f\left(r(t),w(t),\left|\frac{w'(t)}{r'(t)}\right|\right)
:=g(t,w(t),|w'(t)|) \ \text{ on}\ [0,1].
\end{equation*}
If we set $u(t)=u(r(t))$ and $v(t)=v(r(t))$, thus to the system
\eqref{ellbvp-secapp} we can associate the system of ODEs
\begin{equation}\label{1syst}
\begin{cases}
&-u''(t) = g_1(t,u(t),v(t),|u'(t)|,|v'(t)|) \quad \text{ in } [0,1], \\
&-v''(t) =g_2(t,u(t),v(t),|u'(t)|,|v'(t)|) \quad \text{ in } [0,1],\\
& u(0)=u(1)=v(0)=v'(1)=0.
\end{cases}
\end{equation}
We note that, for $i=1,2$, $g_i$  is a positive continuous function in $[0,1]\times[0,+\infty)^4$.\\
 By a \emph{radial solution}  of the system~\eqref{ellbvp-secapp} we mean a solution of the system~\eqref{1syst}.\\
Let $\omega$ be a continuous function on $[0,1]$ and
$$
C_{\omega}^1[0,1]=\{w \in C[0,1]: w \text{ is continuous differentiable on }(0,1) \text{ with }\sup_{t \in (0,1)}\omega(t)|w'(t)|<+\infty\}.
$$
It can  verify that $C_{\omega}^1[0,1]$ is a Banach space (see \cite{aga-ore-yan} for the proof) with the norm 
$$
|| w||: =\max \left\{ || w|| _{\infty},\|w'\|_{\omega}\right\},
$$
where $|| w|| _{\infty}:=\underset{t\in [ 0,\,1]\,}{\max }|w(t)|$ and $\|w'\|_{\omega}:=\displaystyle\sup_{t\in (0,1)}\omega(t)|w'(t)|$.\\
Set $\omega_1(t)=t(1-t)$, $\omega_2(t)=t$, we  study the existence of solutions of the system~\eqref{1syst}
by means of the fixed points of the suitable operator $T$ on the space
$C_{\omega_1}^1[0,1]\times C_{\omega_2}^1[0,1]$ equipped with the norm
\begin{equation*}
|| (u,v)|| :=\max \{|| u|| ,\,|| v|| \}.
\end{equation*}

We define the integral operator $T:C_{\omega_1}^1[0,1]\times C_{\omega_2}^1[0,1]\to C_{\omega_1}^1[0,1]\times C_{\omega_2}^1[0,1]$
\begin{equation}\label{operT}
\begin{array}{c}
T(u,v)(t):=\left(
\begin{array}{c}
T_{1}(u,v)(t) \\
T_{2}(u,v)(t)
\end{array}
\right)  =\left(
\begin{array}{c}
\int_{0}^{1}k_1(t,s)g_{1}(s,u(s),v(s),|u'(s)|,|v'(s)|)\,ds \\
\int_{0}^{1}k_2(t,s)g_2(s,u(s),v(s),|u'(s)|,|v'(s)|)ds%
\end{array}
\right)
\end{array}\,,
\end{equation}
where the Green's functions $k_i$  are given by
\begin{equation*} k_1(t,s)=\begin{cases} s(1-t),&0\leq s\leq t\leq 1,\cr
t(1-s),&0\leq t\leq s\leq 1,\cr\end{cases}
\,\,\,\,\,\,\,k_2(t,s)=\begin{cases} s,&0\leq s\leq t\leq 1,\cr
t,&0 \leq t\leq s\leq 1.\cr\end{cases}\end{equation*}
\smallskip

\noindent
We resume the well known
properties of the Green's functions $k_i$.
\begin{itemize}
 \item[(1)]   The kernel $k_1$ is positive and continuous in  $[0,1]\times  [0,1]$. Moreover, for $[a_1,b_1] \subset(0,1)$, we have
\begin{align*}
k_1(t,s) \leq \phi_1(s)\ \text{ for } (t,s)\in [0,1] \times [0,1]
\text{ and }k_1(t,s) \geq c_1\, \phi_1 (s) \text{ for }(t,s)\in
[a_1,b_1]\times [0,1],
\end{align*}
where we can take
$$\phi_1(s):=\sup_{t\in [0,1]} k_1(t,s)= k_1(s,s)=s(1-s) \text
{ and } c_1:=\min\{a_1, 1-b_1\}\,.
$$
 \item[(2)]  The function $k_1(\cdot,s)$ is derivable in $\tau \in [0,1]$, with
$$
\dfrac{\partial k_1}{\partial t}(t,s)=\begin{cases} -s,&0\leq s<
t\leq 1,\cr 1-s, & 0\leq t< s\leq 1,\cr\end{cases}
$$
and for all $\tau \in [ 0,1]$ we have
\begin{equation*}
\lim_{t\rightarrow \tau }\left| \frac{\partial k_1}{\partial
t}(t,s)- \frac{\partial k_1}{\partial t}(\tau ,s)\right|
=0,\;\text{ for almost every}\,s\in [ 0,1].
\end{equation*}
The partial derivative  $\dfrac{\partial k_1}{\partial t}(t,s)<0$
for $0<s<t$, $\dfrac{\partial k_1}{\partial t}(t,s)>0$ for $t<s<1$
and
$$
 \left|\frac{\partial k_1}{\partial t}(t,s)\right|\leq \psi_1(s):=\max\{s,1-s\}\text{    for   }t\in [0,1]\text{     and almost every }\,s\in [0,1].
 $$
 \item[(3)]   The kernel $k_2$ is positive and continuous in  $[0,1]\times  [0,1]$. Moreover, for $[a_2,b_2] \subset(0,1]$, we have
\begin{align*}
k_2(t,s) \leq \phi_2(s)\ \text{ for } (t,s)\in [0,1] \times [0,1]
\text{ and }k_2(t,s) \geq c_2\, \phi_2 (s) \text{ for }(t,s)\in
[a_2,b_2]\times [0,1],
\end{align*}
where we can take
$$\phi_2(s):=\sup_{t\in [0,1]} k_2(t,s)= k_2(s,s)=s \text{ and }c_2:=a_2\,.
$$
 \item[(4)]  The function $k_2(\cdot,s)$ is derivable in $\tau \in [0,1]$, with
$$
\dfrac{\partial k_2}{\partial t}(t,s)=\begin{cases} 0,&0\leq s<
t\leq 1,\cr 1, & 0\leq t< s\leq 1,\cr\end{cases}
$$
and for all $\tau \in [ 0,1]$ we have
\begin{equation*}
\lim_{t\rightarrow \tau }\left| \frac{\partial k_2}{\partial
t}(t,s)- \frac{\partial k_2}{\partial t}(\tau ,s)\right|
=0,\;\text{ for almost every}\,s\in [ 0,1].
\end{equation*}
Moreover
$$
 \left|\frac{\partial k_2}{\partial t}(t,s)\right|\leq 1:=\psi_2(s)\text{    for   }t\in [0,1]\text{     and almost every }\,s\in [0,1].
 $$
\end{itemize}
By direct calculation  we obtain
$$
m_1:=\left(\sup_{t\in [0,1]}\int_0^1k_1(t,s)\,ds\right)^{-1}=8,
\,\,\,\,\,\, m_2:=\left(\sup_{t\in
[0,1]}\int_0^1k_2(t,s)\,ds\right)^{-1}=2,
$$
$$
M_1:=\left(\inf_{t\in[a_1,b_1]}\int_{a_1}^{b_1}k_1(t,s)ds\right)^{-1}=\begin{cases}\frac{2}{a_1(b_1-a_1)(2-a_1-b_1)},&\mbox
{ if }a_1+b_1\leq
1, \cr\\
\frac{2}{(1-b_1)(b_1^2-a_1^2)},& \mbox {if }a_1+b_1> 1,
\cr\end{cases} \,\,\,\,\,\,
$$
$$
M_2:=\left(\inf_{t\in[a_2,b_2]}\int_{a_2}^{b_2}k_2(t,s)ds\right)^{-1}=\frac{1}{a_2(b_2-a_2)}.
$$
 Fixed $[a_1,b_1] \subset(0,1)$, $c_1=\min\{a_1,1-b_1\}$, $[a_2,b_2] \subset(0,1]$ and $c_2=a_2$, we  consider the
 cones, for $i=1,2$,
\begin{equation*}
K_{i}:=\left\{ w\in C_{\omega_i}^{1}[0,\,1]:w\geq 0,\,\underset{t\in [ a_i,b_i]}%
{\min }w(t)\geq c_i || w|| _{\infty},\,\,  || w|| _{\infty}\geq \|w'\|_{\omega_i}\right\},
\end{equation*}
and the cone $K$ in $ C_{\omega_1}^{1}[0,1]\times C_{\omega_2}^{1}[0,1]$ defined by
\begin{equation*}
K:=\{(u,v)\in K_{1}\times K_{2}\}.
\end{equation*}
By a \emph{positive solution} of the system (\ref{1syst}) we mean
a solution $(u,v)\in K$ of (\ref{1syst}) such that $\|(u,v)\|> 0$.
Note that the functions in $K_i$ are strictly positive on the
sub-intervals $[a_i,b_i]$.\\

The operator $T$ defined in \eqref{operT} leaves the cone $K$ invariant. In fact, take
$(u,v)\in K$ such that $\|(u,v)\| \leq r$, $r>0$, we have, for $t\in [0,1]$,
$$
T_i(u,v)(t) \leq
\int_{0}^{1}\phi_i(s)g_i(s,u(s),v(s),|u'(s)|,|v'(s)|)ds
$$
and so, for $i=1,2$, $|| T_i(u,v)|| _{\infty}<+\infty$.\\
On the other hand, we have
\begin{equation}\label{min}
\min_{t\in [a_i,b_i]}T_i(u,v)(t) \geq
c_i\int_{0}^{1}\phi_i(s)g_i(s,u(s),v(s),|u'(s)|,|v'(s)|)\,ds \geq c_i
||T_i(u,v)|| _{\infty}.
\end{equation}
Now  we prove that for every $(u,v)\in  K$ 
\begin{equation}\label{u'}
\|T_{1}(u,v) \|_{\infty} \geq\|(T_1(u,v))'\|_{\omega_1}.
\end{equation}
We have
\begin{align*}
&t(1-t)|(T_1(u,v))'(t)|=\Big|-t(1-t)\int_0^t
sg_{1}(s,u(s),v(s),|u'(s)|,|v'(s)|)ds\\
&\,\,\,\,\,\,\,\,\,\,\,\,\,\,\,\,\,\,\,\,\,\,\,\,\,\,\,\,\,\,\,\,\,\,\,\,\,\,\,\,\,\,\,\,\,\,\,\,\,\,\,\,\,\,\,\,\,\,+t(1-t)\int_t^1
(1-s)g_{1}(s,u(s),v(s),|u'(s)|,|v'(s)|)ds\Big|\\
&\leq t(1-t)\int_0^tsg_{1}(s,u(s),v(s),|u'(s)|,|v'(s)|)ds+t(1-t)\int_t^1 (1-s)g_{1}(s,u(s),v(s),|u'(s)|,|v'(s)|)ds\\
&\leq (1-t)\int_0^tsg_{1}(s,u(s),v(s),|u'(s)|,|v'(s)|)ds+t\int_t^1 (1-s)g_{1}(s,u(s),v(s),|u'(s)|,|v'(s)|)ds\\
&=T_1(u,v)(t)\leq \|T_1(u,v)\|_{\infty}
\end{align*}
and consequently \eqref{u'} holds.

Moreover we have
\begin{align*}
t|(T_2(u,v)'(t)| &= t\int_{t}^{1}g_{2}(s,u(s),v(s),|u'(s)|,|v'(s)|)ds\le \\
&\leq \int_0^t sg_{2}(s,u(s),v(s),|u'(s)|,|v'(s)|)ds+\int_{t}^{1}tg_{2}(s,u(s),v(s),|u'(s)|,|v'(s)|)ds\\
&=T_2(u,v)(t)\leq \|T_2(u,v)\|_{\infty}
\end{align*}
 and therefore
 \begin{equation}\label{v'}
 \|T_{2}(u,v) \|_{\infty} \geq \|(T_2(u,v))'\|_{\omega_2}.
\end{equation}
 Since (\ref{min} ), (\ref{u'}) and (\ref{v'}) hold for every $r>0$, we obtain  $T_{i}K_{i}\subset K_{i}$. \\
By the properties of the Green's functions $k_i$ and using the
Arz\`{e}la-Ascoli Theorem, we obtain that the operator $T$ is
compact.

In order to use the fixed point index,  we utilize the open bounded sets (relative to $K$) , for
$\rho_1,\rho_2>0$,
\begin{equation*}
K_{\rho _{1},\rho _{2}}:=\{(u,v)\in K:|| u|| <\rho _{1}\ \text{ and }\ || v|| <\rho _{2}\}.
\end{equation*}
\begin{equation*}
 V_{\rho _{1},\rho _{2}}:=\{(u,v)\in K:
\displaystyle{\min_{t\in [a_1,b_1]}}u(t)<\rho_1 \text{ and
}\displaystyle{\min_{t\in [a_2,b_2]}}v(t)<\rho_2\}.
\end{equation*}
The sets defined above have the following properties:
\begin{itemize}
\item[$(P_1)$]$K_{\rho_1,\rho_2}\subset V_{\rho_1,\rho_2}\subset K_{\rho_1/c_1,\rho_2/c_2}$.
\item[$(P_2)$] $(w_1,w_2)\in\partial K_{\rho _{1},\rho _{2}}$ if and only if
$(w_{1},w_{2})\in K$ and for some $i\in \{1,2\}$  $\|w_i\|_{\infty}=\rho_i$ and $c_i \rho_i \le
w_i(t)\le \rho_i$ for $t\in[a_i,b_i]$.
\item[$(P_3)$] $(w_1,w_2) \in \partial V_{\rho_1,\rho_2}$ if and only if  $(w_1,w_2)\in K$ and for some $i\in \{1,2\}$ $\displaystyle
\min_{t\in [a_i,b_i]} w_i(t)= \rho_i$  and
 $\rho_i \le w_i(t) \le \rho_i/c_i$  for  $t\in[a_i,b_i]$.
\end{itemize}
The following theorem follows from classical
results about fixed point index (more details can be seen, for example, in
\cite{Amann-rev, guolak}).
\begin{thm} \label{index}Let $K$ be a cone in an ordered Banach space $X$. Let $\Omega $ be
an open bounded subset with $0 \in \Omega\cap K$ and
$\overline{\Omega \cap K}\neq K$.  Let $\Omega ^{1}$ be open in
$X$ with $\overline{\Omega ^{1}}\subset \Omega \cap K$. Let
$F:\overline{\Omega \cap K}\rightarrow K$ be a compact map.
 Suppose that
\begin{itemize}
\item[(1)]$Fx\neq \mu x$ for all $x\in\partial( \Omega \cap K)$
and for all $\mu \geq 1$.
\item[(2)] There exists $h\in K\setminus \{0\}$ such that $x\neq Fx+\lambda h$ for all $x\in \partial (\Omega^1 \cap K)$ and all $\lambda
>0$.
\end{itemize}
Then $F$ has at least one fixed point $x \in (\Omega \cap
K)\setminus\overline{(\Omega^{1}\cap K)}$.\\
Denoting by $i_K(F,U)$ the fixed point index of $F$ in some
$U\subset X$, we have $$i_{K}(F,\Omega \cap K)=1 \mbox{ and }
i_{K}(F,\Omega^{1} \cap K)=0\,.$$
 The same result holds if
$$i_{K}(F,\Omega \cap K)=0 \mbox{ and }
i_{K}(F,\Omega^{1} \cap K)=1\,.$$
\end{thm}

\section{Existence and non-existence results}
We provide now the existence and non-existence results for the system 
\begin{equation}\label{PDE}
\begin{cases}\
&-\Delta u =f_1(|x|,u,v,|\nabla u|,|\nabla v|)\text{ in } \Omega, \\
&-\Delta v=f_2(|x|,u,v,|\nabla u|,|\nabla v|)\text{ in } \Omega,\\
&u=0 \text{ on }\partial \Omega,\\
&v=0 \text{ on }|x|=R_0 \text{ and }\displaystyle\frac{\partial
v}{\partial r}=0 \text{ on }|x|=R_1,
\end{cases}
\end{equation}
where $\Omega=\{ x\in\mathbb{R}^n : R_0<|x|<R_1\}$ is an annulus,
$0<R_0<R_1<+\infty$ and  for $i=1,2$, $f_i:[R_0,R_1]\times[0,+\infty)^4 \to [0,+\infty)$ is a continuous function.

We define the following
sets:
\begin{align*}
\Omega^{\rho_1,\rho_2}&=[ R_0,R_1]\times \left [0,
\rho_1\right]\times\left [0, \rho_2\right]\times
\left[0, +\infty\right)\times\left [ 0,
+\infty\right),\\
A_1^{s_1,s_2}&=[
\min\{r(a_1),r(b_1)\},\max\{r(a_1),r(b_1)\}]\times\left[s_1,\frac{s_1}{c_1}\right]\times\left[0,\frac{s_2}{c_2}\right]\times\left[0, +\infty\right)\times\left [ 0,
+\infty\right),\\
A_2^{s_1,s_2}&=[\min\{r(a_2),r(b_2)\},\max\{r(a_2),r(b_2)\}]\times\left[0,\frac{s_2}{c_2}\right]\times\left[s_2,\frac{s_2}{c_2}\right]\times\left[0, +\infty\right)\times\left [ 0,
+\infty\right),\\
\end{align*}

\begin{thm}\label{ellyptic}
 Suppose that there exist
$\rho_1, \rho_2, s_1,s_2\in (0,+\infty )$, with $\rho _{i}<c_i\,s
_{i}\,,\,\,i=1,2$, such that the following conditions hold
\begin{equation}\label{pde2}
\sup_{\Omega^{\rho_1,\rho_2}}
f_i(r,w_1,w_2,z_1,z_2)<\frac{m_i}{\displaystyle\sup_{t \in
[0,1]}p(t)}\,\rho_i
\end{equation}
and
\begin{equation}\label{pde3}
\inf_{A_i^{s_1,s_2}}
f_i(r,w_1,w_2,z_1,z_2)>\frac{M_i}{\displaystyle\inf_{t
\in [a_i,b_i]}p(t)}\,s_i.
\end{equation}
Then the system (\ref{PDE}) has at least one positive radial
solution. 
\end{thm}

\begin{proof}
We note that the choice of the numbers $\rho_i$ and $s_i$
assures the compatibility of conditions \eqref{pde2} and \eqref{pde3}.  Moreover, when $f_i$ acts on $\Omega^{\rho_1,\rho_2}$ and $A_i^{s_1,s_2}$,
$g_i$ acts respectively on the subsets  $\tilde{\Omega}^{\rho_1,\rho_2}$ and $\tilde{A}_i^{s_1,s_2}$  defined by
\begin{align*}
\tilde{\Omega}^{\rho_1,\rho_2}&=[ 0,1]\times \left [0,
\rho_1\right]\times\left [0, \rho_2\right]\times \left[0, +\infty\right)\times\left [ 0,
+\infty\right),\\
\tilde{A}_1^{s_1,s_2}&=[
a_1,b_1]\times\left[s_1,\frac{s_1}{c_1}\right]\times\left[0,\frac{s_2}{c_2}\right]\times\left[0, +\infty\right)\times\left [ 0,
+\infty\right),\\
\tilde{A}_2^{s_1,s_2}&=[a_2,b_2]\times\left[0,\frac{s_1}{c_1}\right]\times\left[s_2,\frac{s_2}{c_2}\right]\times\left[0, +\infty\right)\times\left [ 0,
+\infty\right).
\end{align*}
Firstly we claim that $\lambda (u,v)\neq T(u,v)$ for every $(u,v)\in
\partial K_{\rho_1,\rho_2}$ and for every $\lambda \geq 1$,
which implies that the index of $T$ is 1 on $K_{\rho_1,\rho_2}$.

Assume this is not true. Then there exist $\lambda \geq 1$ and $(u,v)\in
\partial K_{\rho_1,\rho_2}$ such that $\lambda (u,v)=T(u,v)$.

Suppose that $\|u\|=|| u || _{\infty}=\rho_1$ holds. Then\begin{align*}
&\lambda u(t)= \int_{0}^{1} k_1(t,s)
g_1(s,u(s),v(s),|u'(s)|, |v'(s)|)ds\\
&=\int_{0}^{1} k_1(t,s)
p(s)f_1\left(r(s),u(s),v(s),\left|\frac{u'(s)}{r'(s)}\right|, \left|\frac{v'(s)}{r'(s)}\right|\right)ds\\
& < m_1 \rho_1\,\int_{0}^{1} \frac{p(s)}{\displaystyle\sup_{t \in[0,1]}p(t)} k_1(t,s)ds\leq m_1\rho_1\,\int_{0}^{1} k_1(t,s)ds.
\end{align*}
Taking the supremum in $[0,1]$, we have
\begin{equation*}
\lambda \rho_1< m_1 \rho_1\sup_{t
\in [0,1]}\int_0^1k_1(t,s)ds= \dfrac{m_1}{m_1} \rho_1=\rho_1.
\end{equation*}
Therefore we obtain $\lambda \rho_1<\rho_1,$ which
contradicts the fact that $\lambda \geq 1$.\\
The case $\|v\|=\|v\|_{\infty}=\rho_2$ analogously follows.\\

Consider $h(t)=1$ for $t\in [ 0,1],$ and note that $(h,h)\in
K$.
Now we claim that
\begin{equation*}
(u,v)\neq T(u,v)+\lambda (h,h)\quad \text{for }(u,v)\in \partial V_{s_1,s_2}\quad \text{and }\lambda \geq 0,
\end{equation*}%
that assures that the index of $T$ is $0$ on $V_{s_1,s_2}$.
Assume, by contradiction, that there exist $(u,v)\in \partial V_{s_1,s_2}$ and $\lambda \geq 0$ such that $(u,v)=T(u,v)+\lambda (h,h)$.
Without loss of generality, we can assume that, for $t\in [a_1,b_1]$,
$$\min_{t\in[a_1,b_1]} u(t)=
s_1,\,s_1\leq u(t)\leq {s_1/c_1}.$$
Then, for $t\in [ a_{1},b_{1}]$, we obtain
\begin{align*}
u(t) &=\int_{0}^{1}k_1(t,s)g_1(s,u(s),v(s),|u'(s)|,|v'(s)|)ds+\lambda h(t) \\
&=\int_{0}^{1}k_1(t,s)p(s)f_1\left(r(s),u(s),v(s),\left|\frac{u'(s)}{r'(s)}\right|, \left|\frac{v'(s)}{r'(s)}\right|\right))ds+\lambda h(t) \\
&\geq \int_{a_1}^{b_1}k_1(t,s)p(s)f_1\left(r(s),u(s),v(s),\left|\frac{u'(s)}{r'(s)}\right|, \left|\frac{v'(s)}{r'(s)}\right|\right))ds \\
&>M_1s_1\int_{a_1}^{b_1}\frac{p(s)}{\displaystyle\inf_{t \in [a_1,b_1]}p(t)}k_{1}(t,s)ds \geq M_1s_1 \int_{a_1}^{b_1}k_{1}(t,s)ds.
\end{align*}
Taking the minimum over $[a_1,b_1]$ gives
\begin{equation*}
s_1=\min_{t\in [ a_1,b_1]}u(t)> M_1s_1\frac{1}{M_1}=s_1,
\end{equation*}
i.e. a contradiction.

The case of $\displaystyle\min_{t\in[a_2,b_2]} v(t)=s_2$ follows in a similar manner.\\
Therefore we have $i_K(T, K_{\rho_1,\rho_2})=1$ and $i_K(T,
V_{s_1,s_2})=0$. From Theorem \ref{index} it follows that the compact operator $T$ has a
fixed point in $V_{s_1,s_2}\setminus
\overline{K}_{\rho_1,\rho_2}$. Then the system \eqref{PDE} admits a positive radial solution. 
\end{proof}
\begin{ex} We note that Theorem ~\ref{ellyptic} can be apply when the nonlinearities $f_i$ are of the type
\begin{equation*}
f_i(|x|,u,v,|\nabla u|, |\nabla v|)=(\delta_i u^{\alpha_i}+\gamma_i v^{\beta_i})q_i(|x|,u,v,|\nabla u|, |\nabla v|)
\end{equation*}
with $q_i$ continuous functions bounded  by a strictly positive constant, $\alpha_i,\beta_i>1$ and suitable $\delta_i,\gamma_i \geq 0$.\\
For example we can consider  in $\mathbb{R}^3$  the system of BVPs
\begin{gather}\label{ellbvpex}
\begin{cases}
&-\Delta u =\frac{e^{-|x|^2}}{6}\,(2-\sin(|\nabla u|^2+|\nabla v|^2)\,u^5 \text{ in } \Omega, \\
&-\Delta v=\frac{1}{\pi}\,e^{-|x|^2}\,\arctan\left(1+|\nabla u|^2+|\nabla v|^2\right)\,v^5\text{ in } \Omega,\\
&u=0 \text{ on }\partial \Omega,\\
&v=0 \text{ on }|x|=1 \text{ and }\displaystyle\frac{\partial
v}{\partial r}=0 \text{ on }|x|=e,
\end{cases}
\end{gather}
where $\Omega=\{ x\in\mathbb{R}^3 : 1<|x|<e\}$.\\
By direct computation, we obtain $\displaystyle\sup_{t \in [0,1]}p(t)=e^2(e-1)^2$ and, fixed $[a_1,b_1]=\left[\frac{1}{4},\frac{3}{4}\right],\,[a_2,b_2]=\left[\frac{1}{2},1\right]$, we have $c_1=\frac{1}{4},\,c_2=\frac{1}{2}$,
\begin{align*}
&\inf_{t \in \left[\frac{1}{4},\frac{3}{4}\right]}p(t)=\frac{e^2(e-1)^2}{\left(e-\frac{e-1}{4}\right)^4},\,\,\,\,\,\,\inf_{t \in \left[\frac{1}{2},1\right]}p(t)=\frac{e^2(e-1)^2}{\left(e-\frac{e-1}{2}\right)^4},\\
&M_1:=\left(\inf_{t\in[a_1,b_1]}\int_{a_1}^{b_1}k_1(t,s)ds\right)^{-1}=16,\,\,\,\,\,\,
M_2:=\left(\inf_{t\in[a_2,b_2]}\int_{a_2}^{b_2}k_2(t,s)ds\right)^{-1}=4.
\end{align*}
With the choice of $\rho_1=\rho_2=1/10,\,\, s_1=s_2=10$,
we obtain
\begin{align*}
&\sup_{\Omega^{\rho_1,\rho_2}}\, f_1\,\leq \frac{\rho_1^5}{2}=  5\times10^{-6}<0.0366701=\frac{m_1}{\displaystyle\sup_{t \in
[0,1]}p(t)}\,\rho_1,\\
&\sup _{\Omega^{\rho_1,\rho_2}} f_2\; \leq \frac{\rho_2^5}{2}=5\times10^{-6}<0.00916753=\frac{m_2}{\displaystyle\sup_{t \in
[0,1]}p(t)}\,\rho_2,\\
&\,\,\,\,\,\,\,\,\,\,\,\,\,\,\,\,\,\,\,\,\,\,\,\,\,\,\,\,\,\,\,\,\,\,\,\,\,\,\,\,\,\,\,\,\,\,\,\,\,\,\,\,\,\,\,\,\,\,\,\,\,\,\,\,\,\,\,\,\,\,\,\,\,\,\,\,\,\,\,\,\,\Omega^{\rho_1,\rho_2}= [1,e]\times\left [0,\frac{1}{10}\right]t\times\left[0,\frac{1}{10}\right]\times [0,+\infty)\times[0,+\infty);\\
&\inf_{A_1^{s_1,s_2}}  f_1\,\geq  \frac{e^{-\left(\frac{4e}{e+3}\right)^2}}{6}s_1^5=448.356>201.236=\frac{M_1}{\displaystyle\inf_{t \in
\left[\frac{1}{4},\frac{3}{4}\right]}p(t)}\,s_1,\\
&\,\,\,\,\,\,\,\,\,\,\,\,\,\,\,\,\,\,\,\,\,\,\,\,\,\,\,\,\,\,\,\,\,\,\,\,\,\,\,\,\,\,\,\,\,\,\,\,\,\,\,\,\,\,\,\,\,\,\,\,\,\,\,\,\,\,\,\,\,\,\,\,\,\,\,\,\,\,\,\,\,A_1^{s_1,s_2}= \left[\frac{4e}{3e+1},\frac{4e}{e+3}\right]\times [10, 40]\times[0,40]\times[0,+\infty)\times[0,+\infty);\\
&\inf_{A_2^{s_1,s_2}}\, f_2\; \geq \frac{2e^{-e^2}\arctan1}{\pi}\,s_2^5=30.8989>21.9044=\frac{M_2}{\displaystyle\inf_{t \in
\left[\frac{1}{2},1\right]}p(t)}\,s_2,\\
&\,\,\,\,\,\,\,\,\,\,\,\,\,\,\,\,\,\,\,\,\,\,\,\,\,\,\,\,\,\,\,\,\,\,\,\,\,\,\,\,\,\,\,\,\,\,\,\,\,\,\,\,\,\,\,\,\,\,\,\,\,\,\,\,\,\,\,\,\,\,\,\,\,\,\,\,\,\,\,\,\,A_2^{s_1,s_2}= \left[\frac{2e}{e+1},e\right]\times [0,20]\times[10,20]\times[0,+\infty)\times[0,+\infty).\\
\end{align*}
Then the hypotheses of Theorem~\ref{ellyptic} are
satisfied and hence the system~\eqref{ellbvpex} has at least one
a positive solution.
\end{ex}
By means of Theorem~\ref{index} and the results contained in Theorem~\ref{ellyptic}, it is possible to obtain results about the existence of  \emph{multiple} positive solutions
of the system~\eqref{PDE}. For brevity, here we state a result on the existence of two positive solutions and refer to~\cite{lan-lin-na,  lanwebb} for the other kind of results that can be stated.
\begin{thm}
Suppose that  there exist $\rho _{i},s _{i},\theta_i\in (0,\infty )$ with $\rho _{i}/c_i<s_i <\theta_{i}$  such that
\begin{align*}
&\inf_{A_i^{\rho_1,\rho_2}}
f_i(r,w_1,w_2,z_1,z_2)>\frac{M_i}{\displaystyle\inf_{t
\in [a_i,b_i]}p(t)}\,\rho_i,\\
&\sup_{\Omega^{s_1,s_2}}
f_i(r,w_1,w_2,z_1,z_2)<\frac{m_i}{\displaystyle\sup_{t \in
[0,1]}p(t)}\,s_i\\
&\inf_{A_i^{\theta_1,\theta_2}}
f_i(r,w_1,w_2,z_1,z_2)>\frac{M_i}{\displaystyle\inf_{t
\in [a_i,b_i]}p(t)}\,\theta_i.
\end{align*}
Then the system~\eqref{PDE} has at least two positive radial
solutions.
\end{thm}
We now show some non-existence results for the system of ellyptic
equations \eqref{PDE} when the nonlinearities $f_i$ are enough "small" or "large".

\begin{thm}
Assume that one of following conditions holds:
\begin{equation}\label{cond1}
f_i(r,w_1,w_2,z_1,z_2)<\frac{m_i}{\displaystyle\sup_{t\in [0,1]}p(t)} w_i\,,\,\, r\in [R_0,R_1],\,
w_i>0, \text{ for }i=1,2,
\end{equation}
\begin{equation}\label{cond2}
f_i(r,w_1,w_2, z_1,z_2)>\frac{M_i}{\displaystyle\inf_{t\in [a_i,b_i]}p(t)} w_i\,\,,\,\, r\in [R_0,R_1] ,\,
w_i>0, \text{ for }i=1,2.
\end{equation}
Then the only possible positive solution of the system \eqref{PDE} is the zero one.
\end{thm}
\begin{proof}
Suppose that (\ref{cond1}) holds and assume that there exists a
solution $(\bar{u},\bar{v})$ of \eqref{PDE}, $(\bar{u},\bar{v})\neq
(0,0)$; then $(u,v):=(\bar{u}\circ r,\bar{v}\circ r)$ is a fixed
point of $T$. Let, for example, be $\|(u,v)\|=\|u\| =\|u\|_{\infty}\neq 0$.\\
 Then, for $t\in [0,1]$, we have
\begin{align*}
&u(t)  =\int_0^1k_1(t,s)g_1(s,u(s),v(s),|u'(s)|,|v'(s)|)ds=\\
&\int_0^1k_1(t,s)p(s)f_1\left(r(s),u(s),v(s),\left|\frac{u'(s)}{r'(s)}\right|,\left|\frac{v'(s)}{r'(s)}\right|\right)ds\\
&<\frac{m_1}{\displaystyle\sup_{t\in [0,1]}p(t)}\int_0^1p(s)k_1(t,s)u(s)ds \\
&\le m_1\int_0^1k_1(t,s)u(s)ds \le m_1\|u\|_{\infty}\int_0^1k_1(t,s)ds.
\end{align*}
 Taking the supremum for $t\in [0,1]$, we have
$$
\|u\|_\infty< m_1\|u\|_\infty\sup_{t\in[0,1]}\int_0^1k_1(t,s)\,ds
=\|u\|_\infty,
$$
a contradiction.\\
Suppose that (\ref{cond2}) holds and assume that there exists
$(u,v)\in K$ such that $(u,v)=T(u,v)$ and $(u,v)\neq (0,0)$. Let,
for example, $\|u\|_{\infty}\neq 0$; then
$\sigma:=\displaystyle\min_{t\in[a_1,b_1]}u(t)>0$ since $u \in K_1$.  We
have, for $t\in [0,1]$,
\begin{align*}&u(t)  =\int_0^1k_1(t,s)g_1(s,u(s),v(s),|u'(s)|,|v'(s)|)ds\\
&\geq\int_{a_1}^{b_1}k_1(t,s)p(s)f_1\left(r(s),u(s),v(s),\left|\frac{u'(s)}{r'(s)}\right|,\left|\frac{v'(s)}{r'(s)}\right|\right)ds\\
&>\frac{M_1}{\displaystyle\inf_{t\in [a_1,b_1]}p(t)}\int_{a_1}^{b_1}p(s)k_1(t,s)u(s)ds 
\geq M_1\int_{a_1}^{b_1}k_1(t,s)u(s)ds.
\end{align*}
Taking the minimum  for $t\in [a_1,b_1]$, we obtain
$$
\sigma=\min_{t\in[a_1,b_1]}u(t)> M_1\inf_{t\in[a_1,b_1]}\int_{a_1}^{b_1}
k_1(t,s)u(s)\,ds \geq M_1\sigma \inf_{t\in[a_1,b_1]}\int_{a_1}^{b_1}k_1(t,s)\,ds
=\sigma,
$$
a contradiction.
\end{proof}


\begin{thebibliography}{00}

\bibitem{aga-ore-yan} 
R. P. Agarwal, D. O'Regan and B. Yan, Multiple positive solutions of singular Dirichlet second order boundary-value problems with derivative dependence, \textit{J. Dyn. Control Syst.}, \textbf{15} (2009), 1--26.

\bibitem{Amann-rev} H. Amann,
Fixed point equations and nonlinear eigenvalue problems in ordered
Banach spaces, \textit{SIAM. Rev.}, \textbf{18} (2009), 620--709.

\bibitem{ave-graef-liu} R. I. Avery, J. R. Graef and X. Liu,
 Compression fixed point theorems of operator type,
 \textit{J. Fixed Point Theory Appl.}, \textbf{17}  (2015), 83--97.

\bibitem{ave-mo-tor} D. Averna, D. Motreanu and E. Tornatore,
Existence and asymptotic properties for quasilinear elliptic equations with gradient dependence,
 \textit{Appl. Math. Lett.},  \textbf{61} (2016), 102--107.

\bibitem{bue-er-zu-fe} H. Bueno, G. Ercole, A. Zumpano and W. M. Ferreira,
 Positive solutions for the $p-$Laplacian with dependence on the gradient,
\textit{Nonlinearity}, \textbf{25} (2012), 1211--1234.

\bibitem{chen-lu} S. Chen and G. Lu,
 Existence and nonexistence of positive radial solutions for a class of semilinear elliptic system,
\textit{Nonlinear Anal.}, \textbf{38} (1999), 919--932.

\bibitem{defig-sa-ubi} D. G.  De Figueiredo, J. S{\'a}nchez and P. Ubilla,
Quasilinear equations with dependence on the gradient,
 \textit{Nonlinear Anal.}, \textbf{71} (2009), 4862--4868.

\bibitem{defig-ubi} D. G. De Figueiredo and P. Ubilla,
Superlinear systems of second-order ODE's,
  \textit{Nonlinear Anal.}, \textbf{68} (2008), 1765--1773.

\bibitem{dolo3} J. M. do {\'O},  S. Lorca, J. S{\'a}nchez, and P. Ubilla,
Positive solutions for a class of multiparameter ordinary elliptic systems, 
\textit{J. Math. Anal. Appl.}, \textbf{332} (2007), 1249--1266.

\bibitem{fa-mi-pe} L. F. O. Faria, O. H.  Miyagaki and F. R. Pereira,
 Quasilinear elliptic system in exterior domains with dependence on the gradient,
   \textit{Math. Nachr.}, \textbf{287} (2014), 61--373.

\bibitem{nirenberg} B. Gidas, W. M. Ni and L. Nirenberg,
 Symmetry and related properties via the maximum principle,
\textit{Comm. Math. Phys.}, {\bf 68} (1979), 209--243.

\bibitem{guo-ge}  Y. Guo and W. Ge,
 Positive solutions for three-point boundary value problems with dependence on the first order derivative,
 \textit{J. Math. Anal. Appl.}, \textbf{290} (2004), 291--301.

\bibitem{guolak} D. Guo and V. Lakshmikantham,
\textit{Nonlinear Problems in Abstract Cones},
Academic Press, Boston, 1988.

\bibitem{inf}
G. Infante and F. Minh{\'o}s,  Nontrivial solutions of systems of Hammerstein integral equations with first derivative dependence,  \text{Mediterr. J. Math.}, \textbf{14} (2017),  18 pp.

\bibitem{janko} T. Jankowski,
Nonnegative solutions to nonlocal boundary value problems for systems of second-order differential equations dependent on the first-order derivatives,
\textit{Nonlinear Anal.}, \textbf{87} (2013), 83â101.

\bibitem{lan-lin-na} K. Q. Lan and W. Lin,
Positive solutions of systems of singular Hammerstein integral equations with applications to semilinear elliptic equations in annuli, 
\textit{Nonlinear Anal.}, \textbf{74} (2011), 7184--7197.

\bibitem{lanwebb} K. Q. Lan and J. R. L. Webb,
Positive solutions of semilinear differential equations with singularities, 
\textit{J. Differential Equations}, \textbf{148} (1998), 407--421.

\bibitem{lee} Y. H. Lee,
  Existence of multiple positive radial solutions for a semilinear elliptic system on an unbounded domain,
   \textit{Nonlinear Anal.}, \textbf{47} (2001), 3649--3660.

\bibitem{lee2} Y. H. Lee,
  Multiplicity of positive radial solutions for multiparameter semilinear elliptic systems on an annulus,
\textit{J. Differential Equations}, \textbf{174} (2001), 420--441.


\bibitem{singh} G. Singh,
 Classification of radial solutions for semilinear elliptic systems with nonlinear gradient terms,
 \textit{Nonlinear Anal.}, \textbf{129} (2015), 77--103.

\bibitem{yang-kong} Z. Yang and L. Kong,
 Positive solutions of a system of second order boundary value problems involving first order derivatives via $\mathbb R^n_+$-monotone matrices, \textit{Nonlinear Anal.}, \textbf{75} (2012), 2037--2046.

\bibitem{zima} M. Zima,
 Positive solutions of second-order non-local boundary value problem with singularities in space variables,
\textit{Bound. Value Probl.}, \textbf{2014} (2014): 200, 9 pp.



\end{thebibliography}
\end{document}